\documentclass[11pt]{article}%
\usepackage{amsmath}
\usepackage{amsfonts}
\usepackage{amssymb}
\usepackage{graphicx}%
\newtheorem{theorem}{Theorem}

\newtheorem{corollary}[theorem]{Corollary}

\newtheorem{lemma}[theorem]{Lemma}

\newenvironment{proof}[1][Proof]{\noindent\textbf{#1.} }{\ \rule{0.5em}{0.5em}}
\begin{document}

\centerline{{\Large {Global Attractivity in Nonlinear Higher Order} }}

\vspace{0.5ex}

\centerline{{\Large {Difference Equations in Banach Algebras} }}

\vspace{2ex}

\centerline{H. SEDAGHAT \footnote{Department of Mathematics, Virginia Commonwealth University, Richmond, Virginia, 23284-2014, USA; Email: hsedagha@vcu.edu}}

\vspace{2ex}

\begin{abstract}
\noindent Non-autonomous, higher order difference equations with linear
arguments of type%
\[
x_{n+1}=\sum_{i=0}^{k}a_{i}x_{n-i}+g_{n}\left(  \sum_{i=0}^{k}b_{i}%
x_{n-i}\right)
\]

\noindent are well-defined on Banach algebras. Their scalar forms with real
variables and parameters have appeared frequently in the literature. By
generalizing existing results from real numbers to algebras and using a new
result on reduction of order, new sufficient conditions are obtained for the
convergence to zero of all solutions of nonlinear difference equations with
linear arguments. Where reduction of order is possible, these conditions
extend the ranges of parameters for which the origin is a global attractor
even when all variables and parameters are real numbers.

\end{abstract}

\bigskip

\section{Introduction}

Special cases of the following type of higher order difference equation have
frequently appeared in the literature in different contexts, both pure and
applied:%
\begin{equation}
x_{n+1}=\sum_{i=0}^{k}a_{i}x_{n-i}+g_{n}\left(  \sum_{i=0}^{k}b_{i}%
x_{n-i}\right)  ,\quad n=0,1,2,\ldots\label{gla}%
\end{equation}

We assume here that $k$ is a fixed positive integer and for each $n$, the
function $g_{n}:\mathbb{X}\rightarrow\mathbb{X}$ is defined on a
real or complex Banach algebra $\mathbb{X}$ with identity. The parameters
$a_{i},b_{i}$ are fixed elements in $\mathbb{X}$ such that
\[
a_{k}\not =0\text{ or }b_{k}\not =0.
\]

Upon iteration, Equation (\ref{gla}) generates a unique sequence of points
$\{x_{n}\}$ in $\mathbb{X}$ (its \textit{solution}) from any given set of
$k+1$ \textit{initial values} $x_{0},x_{-1},\ldots,x_{-k}\in\mathbb{X}$.
The number $k+1$ is the \textit{order} of the difference equation.

Special cases of Equation (\ref{gla}), in the set of real numbers, appeared in
the classical economic models of the business cycle in twentieth century in
the works of Hicks \cite{Hic}, Puu \cite{Puu}, Samuelson \cite{Sam} and
others; see \cite{bk1}, Section 5.1 for some background and references. Other
special cases of (\ref{gla}) occurred later in mathematical studies of
biological models ranging from whale populations to neuron activity; see,
e.g., Clark \cite{Clr}, Fisher and Goh \cite{FG}, Hamaya \cite{Ham} and
Section 2.5 in Kocic and Ladas \cite{KL}.

The dynamics of special cases of (\ref{gla}) with $\mathbb{X=R}$ have been
investigated by several authors. Hamaya uses Liapunov and semicycle methods in
\cite{Ham} to obtain sufficient conditions for the global attractivity of the
origin for the following special case of (\ref{gla})
\[
x_{n+1}=\alpha x_{n}+a\tanh\left(  x_{n}-\sum_{i=1}^{k}b_{i}x_{n-i}\right)
\]
with $0\leq\alpha<1$, $a>0$ and $b_{i}\geq0$. These results can also be
obtained using only the contraction method in \cite{Sgeo}. The results in
\cite{Sgeo} are also used in \cite{bk1}, Section 4.3D, to prove the global
asymptotic stability of the origin for an autonomous special case of
(\ref{gla}) with $a_{i},b_{i}\geq0$ for all $i$ and $g_{n}=g$ for all $n$,
where $g$ is a continuous, non-negative function. The study of global
attractivity and stability of fixed points for other special cases of
(\ref{gla}) appear in \cite{GLV} and \cite{KPS}; also see \cite{KL}, Section 6.9.

The second-order case ($k=1$) has been studied in greater depth. Kent and
Sedaghat obtain sufficient conditions in \cite{KS} for the boundedness and
global asymptotic stability of
\begin{equation}
x_{n+1}=cx_{n}+g(x_{n}-x_{n-1}) \label{sed}%
\end{equation}

Also see \cite{Sed1}. In \cite{Elm}, El-Morshedy improves the convergence
results of \cite{KS} for (\ref{sed}) and also gives necessary and sufficient
conditions for the occurrence of oscillations. The boundedness of solutions of
(\ref{sed}) is studied in \cite{S97} and periodic and monotone solutions of
(\ref{sed}) are discussed in \cite{Sed2}. Li and Zhang study the bifurcations
of solutions of (\ref{sed}) in \cite{LZ}; their results include the
Neimark-Sacker bifurcation (discrete analog of Hopf).

A more general form of (\ref{sed}), i.e., the following equation
\begin{equation}
x_{n+1}=ax_{n}+bx_{n-1}+g_{n}(x_{n}-cx_{n-1}) \label{sed1}%
\end{equation}
is studied in \cite{SKy} where sufficient conditions for the occurrence of
periodic solutions, limit cycles and chaotic behavior are obtained using
reduction of order and factorization of the above difference equation into a
pair of equations of lower order. See \cite{bk2} for some background on order
reduction methods. These methods are used in \cite{dkmos} to determine
sufficient conditions for occurrence of limit cycles and chaos in certain
rational difference equations of the following type%
\[
x_{n+1}=\frac{ax_{n}^{2}+bx_{n-1}^{2}+cx_{n}x_{n-1}+dx_{n}+ex_{n-1}+f}{\alpha
x_{n}+\beta x_{n-1}+\gamma}%
\]
that are special cases of (\ref{sed1}).

In this paper, by generalizing recent results on reduction of order, together
with generalizations of some convergence results from the literature we obtain
sufficient conditions for the global attractivity of the origin for
(\ref{gla}) in the context of Banach algebras. These results also extend
previously known parameter ranges, even in the case of real numbers, i.e.,
$\mathbb{X=R}$ and show that convergence may occur in some cases where the
functions $g_{n}$ or the unfolding map of (\ref{gla}) are not contractions.

Unless otherwise stated, throughout the rest of this paper $\mathbb{X}$ will
denote a real or complex Banach algebra with identity 1 (since
there is very little likelihood of confusion, 1 also denotes the identity of
the underlying field of real or complex numbers). For the basics of Banach
algebras, see, e.g., \cite{Krs} or \cite{Wil}. Each Banach algebra is a Banach
space together with a multiplication operation $xy$ that is associative,
distributes over addition and satisfies the norm inequality
\begin{equation}
|xy|\leq|x||y| \label{ba}%
\end{equation}
with $|1|=1.$ The multiplication by real or complex numbers (scalars) that is
inherited from the vector space structure of $\mathbb{X}$ is made consistent
with the main multiplication by assuming that the following equalities hold
for all scalars $\alpha$:
\[
\alpha(xy)=(\alpha x)y=x(\alpha y).
\]

Elements of type $\alpha1$ where $\alpha$ is a real (or complex) number are
the constants in $\mathbb{X}$. The set $\mathbb{R}$ ($\mathbb{C})$ is a real
(complex) commutative Banach algebra with identity over itself with respect to 
the ordinary addition and multiplication of
complex numbers and the absolute value as norm. The set $C[0,1]$ of all
continuous real valued functions on the interval [0,1] forms a commutative,
real Banach algebra relative to the sup, or max, norm. The identity element is
the constant function $x(r)=1$ for all $r\in\lbrack0,1]$. The other constants
in $C[0,1]$ are just the constant functions on [0,1].

An element $x\in\mathbb{X}$ is invertible, or a unit, if there is $x^{-1}%
\in\mathbb{X}$ (the inverse of $x$) such that $x^{-1}x=1$. The collection of 
all invertible elements
of $\mathbb{X}$ forms a group $\mathcal{G}$ (the group of units) that contains
all nonzero constants. For each $u\in\mathcal{G}$ if $x\in\mathbb{X}$
satisfies the inequality
\[
|x-u|\leq\frac{1}{|u^{-1}|}%
\]
then it can be shown that $x\in\mathcal{G}$. It follows $\mathcal{G}$ is open
relative to the metric topology of $\mathbb{X}$ and contains an open ball of
radius $1/|u^{-1}|$ centered about each $u\in\mathcal{G}$. Since the zero
element is not invertible, $\mathcal{G}\not =\mathbb{X}.$ If $\mathbb{X}$ is
either $\mathbb{R}$ or $\mathbb{C}$ then $\mathcal{G}=\mathbb{X}%
\backslash\{0\}$. In the algebra $C[0,1]$ units are functions that do not
assume the (scalar) value 0.

\section{General results on convergence}

Consider the non-autonomous difference equation%
\begin{equation}
x_{n+1}=f_{n}(x_{n},x_{n-1},\ldots,x_{n-k}) \label{gde}%
\end{equation}
with a given sequence of functions $f_{n}:\mathbb{X}^{k+1}\rightarrow
\mathbb{X}$. We say that the origin is \textit{globally exponentially stable}
if all solutions $\{x_{n}\}$ of (\ref{gde}) in $\mathbb{X}$ satisfy the norm
inequality%
\[
|x_{n}|\leq c^{n}\mu
\]
where $c\in(0,1)$ and $\mu>0$ are real constants such that $c$ is independent
of the initial values $x_{0},x_{-1},\ldots,x_{-k}\in\mathbb{X}$.

The next result, which is true for all Banach spaces (not just algebras)
generalizes Theorem 3 in \cite{Sgeo}.

\begin{lemma}
\label{gas0}Let $\mathbb{X}$ be a Banach space and assume that for some real
$\alpha\in(0,1)$ the functions $f_{n}$ satisfy the norm inequality%
\begin{equation}
|f_{n}(\xi_{0},\xi_{1},\ldots,\xi_{k})|\leq\alpha\max\{|\xi_{0}|,\ldots
,|\xi_{k}|\} \label{lwc}%
\end{equation}
for every $n$ and all $(\xi_{0},\ldots,\xi_{k})\in\mathbb{X}^{k+1}$. Then
every solution $\{x_{n}\}$ of (\ref{gde}) with given initial values
$x_{0},x_{-1},\ldots,x_{-k}\in\mathbb{X}$ satisfies
\[
|x_{n}|\leq\alpha^{n/(k+1)}\max\{|x_{0}|,|x_{-1}|,\ldots,|x_{-k}|\}.
\]
Therefore, the origin is globally exponentially stable.
\end{lemma}

\begin{proof}
Let $\mu=\max\{|x_{0}|,|x_{-1}|,\ldots,|x_{-k}|\}.$ If $\{x_{n}\}$ is the
solution of (\ref{gde}) with the given initial values then we first claim that
$|x_{n}|\leq\alpha\mu$ for all $n\geq1.$ By (\ref{lwc})%
\[
|x_{1}|=|f_{0}(x_{0},x_{-1},\ldots,x_{-k})|\leq\alpha\max\{|x_{0}%
|,\ldots,|x_{-k}|\}=\alpha\mu
\]

and if for any $j\geq1$ it is true that $|x_{n}|\leq\alpha\mu$ for
$n=1,2,\ldots,j$ then%
\begin{align*}
|x_{j+1}|  &  =|f_{j}(x_{j},x_{j-1},\ldots,x_{j-k})|\leq\alpha\max
\{|x_{j}|,|x_{j-1}|,\ldots,|x_{j-k}|\}\\
&  \leq\alpha\max\{\mu,\alpha\mu\}=\alpha\mu.
\end{align*}

Therefore, our claim is true by induction. In particular, since $0<\alpha<1$
we have shown that $|x_{n}|\leq\alpha^{n/(k+1)}\mu$ for $n=1,2,\ldots,k+1.$
Now suppose that $|x_{n}|\leq\alpha^{n/(k+1)}\mu$ is true for $n\leq m$ where
$m\geq k+1.$ Then%
\begin{align*}
|x_{m+1}|  &  =|f_{m}(x_{m},x_{m-1},\ldots,x_{m-k})|\leq\alpha\max
\{|x_{m}|,|x_{m-1}|,\ldots,|x_{m-k}|\}\\
&  \leq\alpha\mu\max\{\alpha^{m/(k+1)},\alpha^{(m-1)/(k+1)},\ldots
,\alpha^{(m-k)/(k+1)}\}\\
&  =\alpha^{(m+1)/(k+1)}\mu
\end{align*}
and the proof is complete by induction.
\end{proof}

\medskip

The above induction argument is used by Berezansky, Braverman and Liz in
\cite{BBL} in the case $\mathbb{X=R}$ and by Xiao and Yang in \cite{XY} in the
autonomous case ($f_{n}=f$ is independent of $n$) for general Banach spaces.
As we see above, this induction argument generalizes to non-autonomous
equations in Banach spaces. Other approaches that yield convergence results
similar to Lemma \ref{gas0} for $\mathbb{X=R}$\ are discussed by Liz in
\cite{Liz}.

For $\mathbb{X=R}$ Lemma \ref{gas0} is also implied by Theorem 2 in
\cite{Mem}\ where Memarbashi uses a contraction argument adapted from Theorem
3 in \cite{Sgeo} (exponential stability, autonomous case in $\mathbb{R}$).
Contraction arguments with their geometric flavor are intuitively appealing
and they also work for non-exponential asymptotic stability; see \cite{Sgeo}
for the autonomous case and \cite{Mem2} which extends the result in
\cite{Sgeo} to certain non-autonomous equations.

For a general Banach space the type of convergence is dictated by the given
norm. For instance, in $C[0,1]$ with the sup, or max, norm convergence to the
zero function in Lemma \ref{gas0} is uniform.

Next, define the following sequence of functions on a Banach algebra
$\mathbb{X}$
\begin{equation}
f_{n}(\xi_{0},\xi_{1},\ldots,\xi_{k})=\sum_{i=0}^{k}a_{i}\xi_{i}+g_{n}\left(
\sum_{i=0}^{k}b_{i}\xi_{i}\right)  \label{fn}%
\end{equation}
The following corollary of Lemma \ref{gas0} generalizes previous convergence
theorems proved for the autonomous case with $\mathbb{X=R}$; e.g., the results
in \cite{Ham} or Theorem 4.3.9(b) in \cite{bk1}.

\begin{lemma}
\label{gas3}Let $g_{n}:\mathbb{X}\rightarrow\mathbb{X}$ be a sequence of
functions on a real or complex Banach algebra $\mathbb{X}$.
Assume that there is a real number $\sigma>0$ such that%
\begin{equation}
|g_{n}(\xi)|\leq\sigma|\xi|,\quad\xi\in\mathbb{X}\label{slg}%
\end{equation}
for all $n$ and further, for coefficients $a_{i},b_{i}$ (real or complex) we
assume that the inequality
\begin{equation}
\sum_{i=0}^{k}(|a_{i}|+\sigma|b_{i}|)<1\label{as}%
\end{equation}
holds. Then every solution $\{x_{n}\}$ of (\ref{gla}) with initial values
$x_{0},x_{-1},\ldots,x_{-k}\in\mathbb{X}$ satisfies%
\[
|x_{n}|\leq\alpha^{n/(k+1)}\max\{|x_{0}|,|x_{-1}|,\ldots,|x_{-k}%
|\},\quad\alpha=\sum_{i=0}^{k}(|a_{i}|+\sigma|b_{i}|).
\]

\end{lemma}

\begin{proof}
If $(\xi_{0},\xi_{1},\ldots,\xi_{k})\in\mathbb{X}^{k+1}$ then by the triangle
inequality, (\ref{ba}) and (\ref{slg})%
\begin{align*}
\left\vert \sum_{i=0}^{k}a_{i}\xi_{i}+g_{n}\left(  \sum_{i=0}^{k}b_{i}\xi
_{i}\right)  \right\vert  &  \leq\sum_{i=0}^{k}(|a_{i}|+\sigma|b_{i}|)|\xi
_{i}|\\
&  \leq\left[  \sum_{i=0}^{k}(|a_{i}|+\sigma|b_{i}|)\right]  \max\{|\xi
_{0}|,\ldots,|\xi_{k}|\}
\end{align*}

Therefore, given (\ref{as}), by Lemma \ref{gas0} the origin is globally
asymptotically stable.
\end{proof}

\medskip

Condition (\ref{slg}) implies that the origin is a fixed point of (\ref{gla})
since it implies that $g_{n}(0)=0$ for all $n$. Except for this restriction,
the functions $g_{n}$ are completely arbitrary. Examples of familiar real
functions $g_{n}$ that satisfy (\ref{slg}) include $\sin t$, $\tan^{-1}t$, and
$\tanh t;$ for instance, $\tanh t$ is used in \cite{Ham}. Of course, $g_{n}$
need not be bounded; e.g., consider $t^{3}/(1+t^{2}).$

\section{Reduction of order}

Under certain conditions a special, \textit{order-reducing} change of
variables splits or factors Equation (\ref{gla}) into a triangular system of
two equations of lower order; see \cite{bk2}, Theorem 5.6. The next lemma
extends that result from fields to algebras.

\begin{lemma}
\label{fsor}Let $g_{n}:\mathbb{X}\rightarrow\mathbb{X}$ be a sequence of
functions on an algebra $\mathbb{X}$ with identity (not necessarily
normed) over a field $\mathcal{F}$. If for $a_{i},b_{i}\in\mathbb{X}$ the
polynomials%
\[
P(\xi)=\xi^{k+1}-\sum_{i=0}^{k}a_{i}\xi^{k-i},\quad Q(\xi)=\sum_{i=0}^{k}%
b_{i}\xi^{k-i}%
\]
have a common root $\rho\in\mathcal{G}$, the group of units of $\mathbb{X}$,
then each solution $\{x_{n}\}$ of (\ref{gla}) in $\mathbb{X}$ satisfies%
\begin{equation}
x_{n+1}=\rho x_{n}+t_{n+1} \label{cfe}%
\end{equation}
where the sequence $\{t_{n}\}$ is the unique solution of the equation:%
\begin{equation}
t_{n+1}=-\sum_{i=0}^{k-1}p_{i}t_{n-i}+g_{n}\left(  \sum_{i=0}^{k-1}%
q_{i}t_{n-i}\right)  \label{fe}%
\end{equation}
in $\mathbb{X}$ with initial values $t_{-i}=x_{-i}-\rho x_{-i-1}\in\mathbb{X}$
for $i=0,1,\ldots,k-1$ and coefficients%
\[
p_{i}=\rho^{i+1}-a_{0}\rho^{i}-\cdots-a_{i}\quad\text{and\quad}q_{i}=b_{0}%
\rho^{i}+b_{1}\rho^{i-1}+\cdots+b_{i}%
\]
in $\mathbb{X}.$ Conversely, if $\{t_{n}\}$ is a solution of (\ref{fe}) with
initial values $t_{-i}\in\mathbb{X}$ then the sequence $\{x_{n}\}$ that it
generates in $\mathbb{X}$ via (\ref{cfe}) is a solution of (\ref{gla}).
\end{lemma}

\begin{proof}
Define the functions $f_{n}$ as in (\ref{fn}) and for every $\xi_{0}%
,v_{1},\ldots,v_{k}$ in $\mathbb{X}$, define $\zeta_{0}=\xi_{0}$ and for
$j=1,\ldots,k$ and fixed $\gamma\in\mathcal{G}$ define%
\[
\zeta_{j}=(\gamma^{-1})^{j}\xi_{0}+\sum_{i=1}^{j}(\gamma^{-1})^{j-i+1}v_{i}.
\]

\noindent Now the change of variables
\begin{equation}
t_{n}=x_{n}-\gamma x_{n-1} \label{cfe0}%
\end{equation}

\noindent in Equation (\ref{gla}) reduces its order by one if and only if the
quantity%
\[
f_{n}(\xi_{0},\zeta_{1},\ldots,\zeta_{k})-\gamma\xi_{0}%
\]
is independent of $\xi_{0}$ (\cite{arx} or \cite{bk2}, Theorem 5.1). In this
case, the above quantity defines a sequence of functions $\phi_{n}%
(v_{1},\ldots,v_{k})$ of $k$ variables that yields a difference equation of
order $k$ (\cite{arx} or \cite{bk2}, Section 5.5). Now, by straightforward
calculation
\begin{align*}
f_{n}(\xi_{0},\zeta_{1},\ldots,\zeta_{k})-\gamma\xi_{0}  &  =(a_{0}-\gamma
)\xi_{0}+\sum_{j=1}^{k}a_{j}\left(  (\gamma^{-1})^{j}\xi_{0}-\sum_{i=1}%
^{j}(\gamma^{-1})^{j-i+1}v_{i}\right)  +\\
&  \qquad+g_{n}\left(  b_{0}\xi_{0}+\sum_{j=1}^{k}b_{j}\left[  (\gamma
^{-1})^{j}\xi_{0}-\sum_{i=1}^{j}(\gamma^{-1})^{j-i+1}v_{i}\right]  \right) \\
&  =\left[  a_{0}-\gamma+\sum_{j=1}^{k}a_{j}(\gamma^{-1})^{j}\right]  \xi
_{0}-\sum_{j=1}^{k}a_{j}\sum_{i=1}^{j}(\gamma^{-1})^{j-i+1}v_{i}+\\
&  \qquad+g_{n}\left(  \left[  b_{0}+\sum_{j=1}^{k}b_{j}(\gamma^{-1}%
)^{j}\right]  \xi_{0}-\sum_{j=1}^{k}b_{j}\sum_{i=1}^{j}(\gamma^{-1}%
)^{j-i+1}v_{i}\right)  .
\end{align*}

The last expression above is independent of $\xi_{0}$ (for arbitrary $\xi_{0}%
$) if and only if $\gamma$ can be chosen such that
\[
a_{0}-\gamma+\sum_{j=1}^{k}a_{j}(\gamma^{-1})^{j}=0\quad\text{and\quad}%
b_{0}+\sum_{j=1}^{k}b_{j}(\gamma^{-1})^{j}=0.
\]

Multiplying the two equalities above on the right by $\gamma^{k}$ yields%
\begin{align*}
0  &  =a_{0}\gamma^{k}-\gamma^{k+1}+\sum_{j=1}^{k}a_{j}(\gamma^{-1})^{j}%
\gamma^{k}=P(\gamma)\\
0  &  =b_{0}\gamma^{k}+\sum_{j=1}^{k}b_{j}(\gamma^{-1})^{j}\gamma^{k}%
=Q(\gamma)
\end{align*}

so that $\gamma$ must be a common root of the polynomials $P$ and $Q.$

Now, let $\gamma=\rho$ be a common root of $P$ and $Q$ in $\mathcal{G}$ and
define the aforementioned functions $\phi_{n}$ as%
\begin{align*}
\phi_{n}(v_{1},\ldots,v_{k})  &  =f_{n}(\xi_{0},\zeta_{1},\ldots,\zeta
_{k})-\rho\xi_{0}\\
&  =-\sum_{j=1}^{k}a_{j}\sum_{i=1}^{j}(\rho^{-1})^{j-i+1}v_{i}+g_{n}\left(
-\sum_{j=1}^{k}b_{j}\sum_{i=1}^{j}(\rho^{-1})^{j-i+1}v_{i}\right) \\
&  =-\sum_{i=1}^{k}\sum_{j=i}^{k}a_{j}(\rho^{-1})^{j-i+1}v_{i}+g_{n}\left(
-\sum_{i=1}^{k}\sum_{j=i}^{k}b_{j}(\rho^{-1})^{j-i+1}v_{i}\right)  .
\end{align*}

For each $i=1,\cdots,k,$ since $\rho$ is a root of the polynomial $P$ it
follows that%
\begin{align*}
\sum_{j=i}^{k}a_{j}(\rho^{-1})^{j-i+1}  &  =\left(  a_{i}\rho^{k-i}%
+a_{i+1}\rho^{k-i-1}+\cdots+a_{k-1}\rho+a_{k}\right)  (\rho^{-1})^{k-i+1}\\
&  =(\rho^{k+1}-a_{0}\rho^{k}-\cdots-a_{i-1}\rho^{k-i+1})(\rho^{-1})^{k-i+1}\\
&  =\rho^{i}-a_{0}\rho^{i-1}-\cdots-a_{i-1}%
\end{align*}

Similarly, since $\rho$ is also a root of the polynomial $Q$ it follows that%
\begin{align*}
\sum_{j=i}^{k}b_{j}(\rho^{-1})^{j-i+1}  &  =\left(  b_{i}\rho^{k-i}%
+b_{i+1}\rho^{k-i-1}+\cdots+b_{k-1}\rho+b_{k}\right)  (\rho^{-1})^{k-i+1}\\
&  =\left(  -b_{0}\rho^{k}-b_{1}\rho^{k-1}-\cdots-b_{i-1}\rho^{k-i+1}\right)
(\rho^{-1})^{k-i+1}\\
&  =-b_{0}\rho^{i-1}-b_{1}\rho^{i-2}-\cdots-b_{i-1}.
\end{align*}

Now, if the quantities $p_{i}$ and $q_{i}$ are defined as in the statement of
this Lemma\textit{ }then the preceding calculations show that%
\[
\sum_{j=i}^{k}a_{j}(\rho^{-1})^{j-i+1}=p_{i-1}\text{ and }\sum_{j=i}^{k}%
b_{j}(\rho^{-1})^{j-i+1}=-q_{i-1}.
\]

Using these quantities the functions $\phi_{n}$ are determined as follows%
\[
\phi_{n}(v_{1},\ldots,v_{k})=-\sum_{i=1}^{k}p_{i-1}v_{i}+g_{n}\left(
\sum_{i=1}^{k}q_{i-1}v_{i}\right)  .
\]

Identifying $v_{i}$ with the new variables $t_{n-i+1}$ yields a difference
equation of order $k$ as follows:%
\begin{align*}
t_{n+1}  &  =\phi_{n}(t_{n},\ldots,t_{n-k+1})\\
&  =-\sum_{i=1}^{k}p_{i-1}t_{n-i+1}+g_{n}\left(  \sum_{i=1}^{k}q_{i-1}%
t_{n-i+1}\right) \\
&  =-\sum_{i=0}^{k-1}p_{i}t_{n-i}+g_{n}\left(  \sum_{i=0}^{k-1}q_{i}%
t_{n-i}\right)
\end{align*}
which is Equation (\ref{fe}). From (\ref{cfe0}) we obtain (\ref{cfe}) and the
proof is complete. \end{proof}

\medskip

\textbf{Remarks}. 1. The preceding result shows that Equation (\ref{gla})
splits into the equivalent pair of equations (\ref{cfe}) and (\ref{fe}) via
the change of variables (\ref{cfe0}) provided that the polynomials $P$ and $Q$
have a common nonzero root $\rho$ in the group of units of $\mathbb{X}$.
Equation (\ref{fe}), which is of the same type as (\ref{gla}) but with order
reduced by one is the\textit{ factor equation} of (\ref{gla}). Equation
(\ref{cfe}) which bridges the order (or dimension) gap between (\ref{gla}) and
(\ref{fe}) is the \textit{cofactor equation}.

2. In the special case where $b_{i}=0$ for all $i$ (\ref{gla}) reduces to the
linear non-homogeneous difference equation%
\begin{equation}
x_{n+1}=\sum_{i=0}^{k}a_{i}x_{n-i}+g_{n}(0) \label{glal}%
\end{equation}
with constant coefficients. Since in this case $Q$ is just the zero
polynomial, Lemma \ref{fsor} is applicable with $\rho$ being any root of $P$
in $\mathcal{G}$. So, as might be expected, when $\mathbb{X=R}$ the reduction
of order is possible if the homogeneous part of (\ref{glal}) has nonzero
eigenvalues. The reduction of order of general linear equations
(non-homogeneous, non-autonomous) on arbitrary fields is discussed in
\cite{bk2}, Chapter 7.

3. Solution of polynomials by factorization is problematic in a general Banach
algebra $\mathbb{X}$ due to limitations of the cancellation law. Requiring all
nonzero elements of $\mathbb{X}$ to be units reduces $\mathbb{X}$ to either
$\mathbb{R}$ or $\mathbb{C}$ in commutative cases and to quaternions in
non-commutative cases; see (\cite{Wil}). For general Banach algebras it
is possible to determine, for special cases of (\ref{gla}), whether $P$ and
$Q$ have a common root in $\mathcal{G}$; see the two corollaries in the next section.

4. Equation (\ref{fe}) preserves another aspect of (\ref{gla}). If
$a_{i},b_{i}\in\{0\}\cup\mathcal{G}$ and $P$ and $Q$ have a common root $\rho$
in $\mathcal{G}$ then the numbers $p_{i},q_{i}$ in Lemma \ref{fsor} are also
units (or zero).

\section{Extending the ranges of parameters}

The solution of Equation (\ref{cfe}) in terms of $t_{n}$ is%
\begin{equation}
x_{n}=\rho^{n}x_{0}+\sum_{j=1}^{n}\rho^{n-j}t_{j} \label{cs}%
\end{equation}

This formula may be used to translate various properties of a solution
$\{t_{n}\}$ of (\ref{fe}) into corresponding properties of the solution
$\{x_{n}\}$ of (\ref{gla}). This is done for equation (\ref{sed1}) in
\cite{SKy} for $\mathbb{X=R}$. Doing the same for (\ref{gla}) more generally
yields the following natural consequence of combining Lemmas \ref{gas3} and
\ref{fsor}.

\begin{theorem}
\label{gas2}Let $g_{n}:\mathbb{X}\rightarrow\mathbb{X}$ be a sequence of
functions that satisfy (\ref{slg}) for each $n$. Then every solution
$\{x_{n}\}$ of (\ref{gla}) converges to zero if either (a) or (b) below is true:

(a) Inequality (\ref{as}) holds;

(b) The polynomials $P,Q$ in Lemma \ref{fsor} have a common root $\rho
\in\mathcal{G}$ such that $|\rho|<1$ and
\begin{equation}
\sum_{i=0}^{k-1}(|p_{i}|+\sigma|q_{i}|)<1 \label{spq}%
\end{equation}
with the coefficients $p_{i},q_{i}$ in the factor equation (\ref{fe}),
$i=0,1,\ldots,k-1$.
\end{theorem}

\begin{proof}
(a) Convergence in this case is an immediate consequence of Lemma \ref{gas3}.

(b) By an application of Lemma \ref{fsor} we obtain (\ref{fe}). Then, given
(\ref{spq}), an application of Lemma \ref{gas3} to (\ref{fe}) implies that
\[
|t_{n}|\leq\alpha^{n/(k+1)}\mu
\]

\noindent where $\mu=\max\{|t_{0}|,|t_{-1}|,\ldots,|t_{-k+1}|\}$ with
$t_{-i}=x_{-i}-\rho x_{-i-1}$ for $i=0,1,\ldots,k-1$ and%
\[
\alpha=\sum_{i=0}^{k-1}(|p_{i}|+\sigma|q_{i}|),\quad
\]

\noindent with $p_{i},q_{i}$ as given in Lemma \ref{fsor}. Since $|\rho
^{j}|\leq|\rho|^{j}$ for each $j$, taking norms in (\ref{cs}) yields%
\begin{equation}
|x_{n}|\leq|\rho|^{n}|x_{0}|+\sum_{j=1}^{n}|\rho|^{n-j}|t_{j}|\leq|\rho
|^{n}|x_{0}|+\mu|\rho|^{n}\sum_{j=0}^{n-1}\left(  \frac{\alpha^{1/(k+1)}%
}{|\rho|}\right)  ^{j}. \label{ms}%
\end{equation}

If $\alpha^{1/(k+1)}\not =|\rho|$ then%
\begin{align*}
|x_{n}|  &  \leq|\rho|^{n}|x_{0}|+\mu\alpha^{1/(k+1)}|\rho|^{n-1}\frac
{[\alpha^{1/(k+1)}/|\rho|]^{n}-1}{[\alpha^{1/(k+1)}/|\rho|]-1}\\
&  =|\rho|^{n}|x_{0}|+\mu\alpha^{1/(k+1)}\frac{\alpha^{n/(k+1)}-|\rho|^{n}%
}{\alpha^{1/(k+1)}-|\rho|}%
\end{align*}

Since $\alpha,|\rho|<1$ it follows that $\{x_{n}\}$ converges to zero. If
$\alpha^{1/(k+1)}=|\rho|$ then (\ref{ms}) reduces to%
\[
|x_{n}|\leq|\rho|^{n}|x_{0}|+\mu|\rho|^{n}n
\]

\noindent and by L'Hospital's rule $\{x_{n}\}$ again converges to zero.
\end{proof}

\begin{corollary}
\label{gasc1}Let $g_{n}$ be functions on $\mathbb{X}$ satisfying (\ref{slg})
for all $n\geq0.$ Every solution of the difference equation
\begin{align}
x_{n+1}  &  =ax_{n}+g_{n}\left(  b_{0}x_{n}+b_{1}x_{n-1}+\cdots+b_{k}%
x_{n-k}\right)  ,\label{gla0}\\
a  &  \in\mathcal{G},b_{i}\in\mathbb{X},\ b_{k}\not =0\nonumber
\end{align}
converges to zero if $|a|<1$ and the following conditions hold:
\begin{align}
b_{0}a^{k}+b_{1}a^{k-1}+b_{2}a^{k-2}+\cdots+b_{k}  &  =0\text{, }\label{cro}\\
\sum_{i=0}^{k-1}|b_{0}a^{i}+b_{1}a^{i-1}+\cdots+b_{i}|  &  <\frac{1}{\sigma}.
\label{cre}%
\end{align}

\end{corollary}

\begin{proof}
For equation (\ref{gla0}) the polynomials $P,Q$ are%
\[
P(\xi)=\xi^{k+1}-a\xi^{k},\quad Q(\xi)=b_{0}\xi^{k}+b_{1}\xi^{k-1}%
+\cdots+b_{k}.
\]

Thus $\rho=a$ is their common root in $\mathcal{G}$ if (\ref{cro}) holds. The
numbers $p_{i},q_{i}$ that define the factor equation (\ref{fe}) in this case
are%
\[
p_{i}=\rho^{i+1}-a\rho^{i}=0,\quad q_{i}=b_{0}a^{i}+b_{1}a^{i-1}+\cdots+b_{i}%
\]

\noindent Thus, if $|a|<1$ then by Theorem \ref{gas2} every solution of
(\ref{gla0}) converges to zero.
\end{proof}

\medskip

\textbf{Remark.} The parameter range determined by (\ref{spq}) is generally
distinct from that given by (\ref{as}); hence, Theorem \ref{gas2} or Corollary
\ref{gasc1} may imply convergence to 0 when Lemma \ref{gas3} does not apply
and the unfolding map is not a contraction. To illustrate, consider the
following equation on the set of real numbers:
\begin{equation}
x_{n+1}=ax_{n}+\alpha_{n}\tanh\left(  x_{n}-bx_{n-k}\right)  \label{dham}%
\end{equation}
which is a non-autonomous version of a type of equation discussed in
\cite{Ham}. Suppose that the sequence $\{\alpha_{n}\}$ of real numbers is
bounded by $\sigma>0$ and it is otherwise arbitrary. Then
\[
|\alpha_{n}\tanh t|=|\alpha_{n}||\tanh t|\leq|\alpha_{n}||t|\leq\sigma|t|
\]
for all $n$ and (\ref{slg}) holds. If $0<|a|<1$ and $b=a^{k}$ then by
Corollary \ref{gasc1} every solution of (\ref{dham}) converges to zero if
\[
\frac{1}{\sigma}>\sum_{i=0}^{k-1}|a|^{i}=\frac{1-|a|^{k}}{1-|a|}%
\]
i.e., if
\begin{equation}
\sigma<\frac{1-|a|}{1-|a|^{k}}. \label{er}%
\end{equation}

On the other hand, applying Lemma \ref{gas3} to (\ref{dham}) with $b=a^{k}$
produces the range%
\[
|a|+\sigma(1+|a|^{k})<1\Rightarrow\sigma<\frac{1-|a|}{1+|a|^{k}}%
\]
which is clearly more restricted than the one given by (\ref{er}). Note that
$a$ and $\sigma$ may satisfy (\ref{er}) but with
\[
|a|+\sigma(1+|a|^{k})>1.
\]

\medskip

\begin{corollary}
\label{gasc2}Let $g_{n}$ be functions on $\mathbb{X}$ satisfying (\ref{slg})
for all $n\geq0.$ For the difference equation%
\begin{align}
x_{n+1}  &  =a_{0}x_{n}+a_{1}x_{n-1}+\cdots+a_{k}x_{n-k}+g_{n}\left(
x_{n}-bx_{n-1}\right)  ,\label{gla1}\\
a_{i}  &  \in\mathbb{X},b\in\mathcal{G},\ a_{k}\not =0\nonumber
\end{align}
assume that $|b|<1$ and the following conditions hold:%
\begin{gather}
a_{0}b^{k}+a_{1}b^{k-1}+\cdots+a_{k}=b^{k+1},\label{cro1}\\
\sum_{i=0}^{k-1}|b^{i+1}-a_{0}b^{i}-\cdots-a_{i}|<1-\sigma. \label{cre1}%
\end{gather}
Then every solution of (\ref{gla1}) converges to zero.
\end{corollary}

\begin{proof}
The polynomials $P,Q$ in this case are%
\[
P(\xi)=\xi^{k+1}-a_{0}\xi^{k}-\cdots-a_{k},\quad Q(\xi)=\xi^{k}-b\xi^{k-1}.
\]

Clearly, $Q(b)=0$ and if Equality (\ref{cro1}) holds then $P(b)=0$ too, so
Theorem \ref{gas2} applies. We calculate the coefficients of the factor
equation (\ref{fe}) as $q_{0}=1$, $q_{i}=0$ if $i\not =0$ and
\[
p_{i}=b^{i+1}-a_{0}b^{i}-\cdots-a_{i}.
\]

Now, inequality (\ref{spq}) yields (\ref{cre1}) via a straightforward
calculation:
\begin{align*}
1  &  >\sum_{i=0}^{k-1}|b^{i+1}-a_{0}b^{i}-\cdots-a_{i}|+\sigma\\
&  =\sum_{i=0}^{k-1}|b^{i+1}-a_{0}b^{i}-\cdots-a_{i}|+\sigma
\end{align*}

Thus, if $|b|<1$ then by Theorem \ref{gas2} every solution of (\ref{gla1})
converges to zero.
\end{proof}

\medskip

As an application of the preceding corollary, consider the case $k=1$, i.e.,
the second-order equation%
\begin{equation}
x_{n+1}=a_{0}x_{n}+a_{1}x_{n-1}+g_{n}(x_{n}-bx_{n-1}) \label{gla2}%
\end{equation}
which is essentially Equation (\ref{sed1}) on a Banach algebra $\mathbb{X}$.
By Corollary \ref{gasc2}, every solution of (\ref{gla2}) converges to zero if
the functions $g_{n}$ satisfy (\ref{slg}) and
\begin{equation}
b\in\mathcal{G},\quad|b|<1,\quad a_{0}b+a_{1}=b^{2},\quad|a_{0}-b|+\sigma<1.
\label{wcor}%
\end{equation}

On the other hand, according to Lemma \ref{gas3}, every solution of
(\ref{gla2}) converges to zero if the functions $g_{n}$ satisfy (\ref{slg})
and%
\begin{equation}
|a_{0}|+|a_{1}|+\sigma(1+|b|)<1. \label{wc}%
\end{equation}

Parameter values that do not satisfy (\ref{wc}) may satisfy (\ref{wcor}). For
comparison, if $a_{1}=b^{2}-a_{0}b$ then (\ref{wc}) may be solved for $\sigma$
to obtain%
\[
\sigma<\frac{1-|a_{0}|-|b||a_{0}-b|}{1+|b|}.
\]

This is a stronger constraint on $\sigma$ than $\sigma<1-|a_{0}-b|$ from
(\ref{wcor}), especially if $b$ is not near 0.

\medskip

\textbf{Example}. Consider the following difference equation on the real
Banach algebra $C[0,1]$ of all continuous functions on the interval [0,1] with
the sup, or max norm:%
\begin{equation}
x_{n+1}=\frac{\alpha r}{r+1}x_{n}+\frac{\beta(\beta-\alpha r)}{(r+1)^{2}%
}x_{n-1}+\int_{0}^{r}\phi_{n}\left(  x_{n}-\frac{\beta}{r+1}x_{n-1}\right)  dr
\label{glac}%
\end{equation}
where the functions $\phi_{n}:\mathbb{R}\rightarrow\mathbb{R}$ are integrable
and for each $n$ they satisfy the absolute value inequality%
\[
|\phi_{n}(r)|\leq\sigma|r|,\quad r\in\mathbb{R}%
\]
for some $\sigma>0$. Assume that the following inequalities hold%
\[
0<\beta<1,\quad3\beta\leq\alpha<2+\beta,\quad\sigma<\frac{2+\beta-\alpha}{2}%
\]
and define the coefficient functions%
\[
a_{0}(r)=\frac{\alpha r}{r+1},\quad a_{1}(r)=\frac{\beta(\beta-\alpha
r)}{(r+1)^{2}},\quad b(r)=\frac{\beta}{r+1}.
\]

Then $b\in\mathcal{G}$, $b^{2}-a_{0}b=\beta(\beta-\alpha r)/(r+1)^{2}=a_{1}$
\noindent and the following are true about the norms:%
\[
|a_{0}|=\sup_{0\leq r\leq1}\frac{\alpha r}{r+1}=\frac{\alpha}{2},\quad
|b|=\sup_{0\leq r\leq1}\frac{\beta}{r+1}=\beta<1
\]

\noindent and%
\[
|a_{0}-b|=\sup_{0\leq r\leq1}\left\vert \frac{\alpha r-\beta}{r+1}\right\vert
=\max\left\{  \frac{\alpha-\beta}{2},\beta\right\}  =\frac{\alpha-\beta}%
{2}<1-\sigma.
\]

It follows that the conditions in (\ref{wcor}) are met. Next, since the
functions $g_{n}:C[0,1]\rightarrow C[0,1]$ in (\ref{gla2}) are defined as
$g_{n}(x)(r)=\int_{0}^{r}(\phi_{n}\circ x)(r)dr$ for $r\in\lbrack0,1]$ and all
$n$, their norms satisfy%
\[
|g_{n}(x)|\leq\sup_{0\leq r\leq1}\int_{0}^{r}|\phi_{n}(x(r))|dr\leq\sup_{0\leq
r\leq1}\int_{0}^{r}\sigma|x(r)|dr\leq\sigma|x|\sup_{0\leq r\leq1}r=\sigma|x|.
\]

Therefore, Corollary \ref{gasc2} may be applied to conclude that for every
pair of initial functions $x_{0},x_{-1}\in C[0,1]$ the sequence of functions
$x_{n}=x_{n}(r)$ that satisfy (\ref{glac}) in $C[0,1]$\ converges uniformly to
the zero function. In addition, it is worth observing that $a_{0}%
\not \in \mathcal{G}$ and $|a_{0}|\geq1$ if $\alpha\geq2$, in which case
(\ref{wc}) does not hold.

\medskip

\textbf{Remark}. The preceding corollaries and Theorem \ref{gas2} are broad
applications of the reduction of order method to very general equations that
improve the range of parameters compared to Lemma \ref{gas3}. They show that
different patterns of delays may be translated into algebraic problems about
the polynomials $P$ and $Q$ and their root structures. In some cases a more
efficient application of Lemma \ref{fsor} yields a greater amount of
information about the behavior of solutions than Theorem \ref{gas2} provides.
The next result represents a deeper use of order reduction in that sense.

\begin{theorem}
\label{21}In Equation (\ref{gla2}) assume that $b\in\mathcal{G}$ with $|b|<1$
and $a_{0},a_{1}\in\mathbb{X}$ such that $a_{0}b+a_{1}=b^{2}.$ If
$x_{0},x_{-1}$ are given initial values for (\ref{gla2}) for which the
solution of the first order equation%
\begin{equation}
t_{n+1}=(a_{0}-b)t_{n}+g_{n}(t_{n}) \label{gla21}%
\end{equation}
converges to zero with the initial value $t_{0}=x_{0}-bx_{-1}$ then the
corresponding solution of (\ref{gla2}) converges to zero. In particular, if
the origin is a global attractor of the solutions of the first order Equation
(\ref{gla21}) then it is also a global attractor of the solutions of
(\ref{gla2}).
\end{theorem}

\begin{proof}
In this case $Q(\xi)=\xi-b$ so there is only one root $b.$ Now Lemma
\ref{fsor} gives Equation (\ref{gla21}) if $P(b)=0,$ i.e., if $a_{0}%
b+a_{1}=b^{2}.$ Finally, we complete the proof by arguing similarly to the
proof of Theorem \ref{gas2}(b), using (\ref{cs}).
\end{proof}

\medskip

\textbf{Example}. consider the following autonomous equation on the real
numbers%
\begin{equation}
x_{n+1}=ax_{n}+b(b-a)x_{n-1}+\sigma\tanh\left(  x_{n}-bx_{n-1}\right)
\label{th}%
\end{equation}
where $\sigma>0$, $0<b<1$ and $a<b.$ Equation (\ref{gla21}) in this case is%
\begin{equation}
t_{n+1}=h(t_{n})\quad h(\xi)=(a-b)\xi+\sigma\tanh\xi. \label{th1}%
\end{equation}

The function $h$ has a fixed point at the origin since $h(0)=0.$ Further, the
origin is the unique fixed point of $h$ if $|h(\xi)|<|\xi|$ for $\xi\not =0.$
Since $h$ is an odd function, it is enough to consider $\xi>0$. In this case,
$h(\xi)<\xi$ if and only if $\sigma\tanh\xi<(1-a+b)\xi.$ Since $\tanh\xi<\xi$
for $\xi>0$ we may conclude that%
\begin{equation}
\sigma<1-a+b. \label{sig}%
\end{equation}

Given that $a<b$, it is possible to choose $1\leq\sigma<1-a+b$ and extend the
range of $\sigma$ beyond what is possible with (\ref{wcor}) or (\ref{wc}),
which require that $\sigma<1$. In particular, the function $\sigma\tanh\xi$ is
not a contraction near the origin in this discussion.

Routine analysis of the properties of $h$ leads to the following bifurcation scenario:

\begin{enumerate}
\item Suppose that $b-1\leq a<b<1$ and (\ref{sig}) holds. Then $b-a\leq1$ and
all solutions of (\ref{th1}), hence, also all solutions of (\ref{th}) converge
to zero.

\item Now we fix $b,\sigma$ and reduce the value of $a$ so that $a<b-1<0$.
Then the function $h\circ h$ crosses the diagonal at two points $\tau>0$ and
$-\tau$ and a 2-cycle $\{-\tau,\tau\}$ emerges for Equation (\ref{th1}). Note
that (\ref{sig}) still holds when $a$ is reduced, but the origin is no longer
globally attracting. The cycle $\{-\tau,\tau\}$ is repelling and generates a
repelling 2-cycle for (\ref{th}); see \cite{SKy} or \cite{bk2}, Section 5.5.
The emergence of this cycle implies that $\{t_{n}\}$ is unbounded if
$|t_{0}|>\tau$ and it converges to 0 if $|t_{0}|<\tau.$ Therefore, the
corresponding solution $\{x_{n}\}$ of (\ref{th}) also converges to 0 if
\[
|x_{0}-bx_{-1}|=|t_{0}|<\tau;
\]
\noindent i.e., if the initial point $(x_{-1},x_{0})$ is between the two
parallel lines $y=b\xi+\tau$ and $y=b\xi-\tau$ in the $(\xi,y)$ plane.

\item Suppose that $a$ continues to decrease. Then the value of $\tau$ also
decreases and reaches zero when
\[
a=b-\sigma-1
\]
i.e., when the slope of $h$ at the origin is $-1$. Now, the cycle
$\{-\tau,\tau\}$ collapses into the origin and turns it into a repelling fixed
point. In this case, \textit{all} nonzero solutions of (\ref{th1}) and
(\ref{th}) are unbounded.
\end{enumerate}

\end{document}